\documentclass[twoside,8pt,B5]{article}

\usepackage{graphicx,amsmath,latexsym,amssymb,amsthm,fancyhdr}

\font\chuto=cmbx10 at 16pt \font\chudam=cmbxsl10 \font\chudams=cmbxsl8

\newtheorem{thm}{Theorem}[section]
\newtheorem{cor}[thm]{Corollary}

\newtheorem{prop}[thm]{Proposition}
\theoremstyle{definition}

\newtheorem{rem}[thm]{Remark}
\newtheorem{ex}[thm]{Example}
\newtheorem{conj}[thm]{Conjecture}
\numberwithin{equation}{section}

\fancypagestyle{plain}{%
\fancyhead[L]{\setlength{\unitlength}{1cm}
\begin{picture}(0,0)
\put(-0.2,1.0){\chudam International Conference in Number Theory and Applications 2012}
\put(-0.2,0.5){\chudams Department of Mathematics, Faculty of Science, Kasetsart University}
\put(-0.2,0.0){\chudams Speaker: M.~Mattila}  
\put(-.2,-0.2){\line(1,0){11.3}}
\put(-.2,-0.27){\line(1,0){11.3}}
\put(-.2,-0.3){\line(1,0){11.3}}
\put(-.2,-0.32){\line(1,0){11.3}}
\end{picture}\hfill \thepage}
\fancyfoot[C]{} 

}
\newcommand*{\Rset}{\mathbb{R}}

\newcommand*{\Zset}{\mathbb{Z}}
\newcommand*{\Cset}{\mathbb{C}}
\newcommand*{\ltriplevert}{\left|\mkern-2mu\left|\mkern-2mu\left|}
\newcommand*{\rtriplevert}{\right|\mkern-2mu\right|\mkern-2mu\right|}
\newcommand*{\matrixnorm}[1]{\ltriplevert #1\rtriplevert}

\begin{document}

\setcounter{page}{1}

\thispagestyle{plain}
\centerline {~}
\centerline {\bf \chuto  On the eigenvalues of certain}
\medskip
\centerline {\bf \chuto  number-theoretic matrices}

\fancyhead[RE]{\small{\thepage \hfill \textit{M.~Mattila and P.~Haukkanen}}}  
\fancyhead[LO]{\small{\textit{On the eigenvalues of certain number-theoretic matrices} \hfill \thepage}}

\renewcommand{\thefootnote}{\fnsymbol{footnote}}

\vskip.8cm
\centerline {Mika Mattila$^{1,}${\footnote{\textit{Corresponding author}}} and
             Pentti Haukkanen$^{2}$
}

\renewcommand{\thefootnote}{\arabic{footnote}}
\vskip.5cm
\centerline{School of Information Sciences}
\centerline{FI-33014 University of Tampere, Finland}
\vskip.5cm
\centerline{$^1$author1's address}
\centerline{\texttt{mika.mattila@uta.fi}
}
\centerline{$^2$author2's address
}
\centerline{\texttt{pentti.haukkanen@uta.fi}
}

\vskip .5cm

\begin{abstract}
In this paper we study the structure and give bounds for the eigenvalues of the $n\times n$ matrix, which $ij$ entry is $(i,j)^\alpha[i,j]^\beta$, where $\alpha,\beta\in\Rset$, $(i,j)$ is the greatest common divisor of $i$ and $j$ and $[i,j]$ is the least common multiple of $i$ and $j$. Currently only $O$-estimates for the greatest eigenvalue of this matrix can be found in the literature, and the asymptotic behaviour of the greatest and smallest eigenvalue is known in case when $\alpha=\beta$.
\end{abstract}

\medskip






\section{Introduction}

Let $S=\{x_1,x_2,\ldots,x_n\}$ be a set of distinct positive integers, and let $f$ be an arithmetical function. Let $(S)_f$ denote the $n\times n$ matrix having $f$ evaluated at the greatest common divisor $(x_i,x_j)$ of $x_i$ and $x_j$ as its $ij$ entry. More formally, let $((S)_f)_{ij}=f((x_i,x_j))$. Analogously, let $[S]_f$ denote the $n\times n$ matrix having $f$ evaluated at the least common multiple $[x_i,x_j]$ of $x_i$ and $x_j$ as its $ij$ entry. That is, $([S]_f)_{ij}=f([x_i,x_j])$. The matrices $(S)_f$ and $[S]_f$ are referred to as the GCD and LCM matrices on $S$ associated with $f$.

The study of GCD and LCM matrices was initiated by H. J. S. Smith \cite{S} in 1875, when he calculated $\det(S)_f$ in case when $S$ is factor-closed and $\det[S]_f$ in a more special case. Since Smith, numerous papers have been published about GCD and LCM matrices. For general accounts see e.g. \cite{HWS, SC}. There are also various generalizations of GCD and LCM matrices to be found in the literature. The most important ones are the lattice-theoretic generalizations into meet and join matrices, see e.g. \cite{KH}.

Over the years some authors have studied number-theoretic matrices that are neither GCD nor LCM matrices, but very closely related to them. For example, Wintner \cite{W} published results concerning the largest eigenvalue of the $n\times n$ matrix having
\[
\left(\frac{(i,j)}{[i,j]}\right)^\alpha
\]
as its $ij$ entry and subsequently Lindqvist and Seip \cite{LS} investigated the asymptotic behavior of the smallest and largest eigenvalue of the same matrix. More recently Hilberdink \cite{Hi} as well as Berkes and Weber \cite{BW} have studied this same topic from analytical perspective.

Also the norms of GCD, LCM and related matrices have been repeatedly studied in the literature. Altinisik et al. \cite{ATH} investigated the norms of reciprocal LCM matrices, and later Altinisik \cite{A} published a paper about the norms of GCD related matrix. Haukkanen \cite{Ha1, Ha3, Ha2} studied the $n\times n$ matrix having
\[
\frac{(i,j)^r}{[i,j]^s},\quad r,s\in\Rset,
\]
as its $ij$ entry and, among other things, gave $O$-estimates for the $\ell_p$ and maximum row and column sum norms of this matrix. In this paper we study the same class of matrices, although we use a slightly different notation. Let $\alpha,\beta\in\Rset$. Our goal here is to find bounds for the eigenvalues of the $n\times n$ matrix having
\[
(i,j)^\alpha[i,j]^\beta
\]
as its $ij$ entry. In order to do this we use similar techniques as Ilmonen et al \cite{IHM} and Hong and Loewy \cite{HL2}. One of the methods may be considered to originate from Hong and Loewy \cite{HL}. It should be noted that not much is known about the eigenvalues of GCD, LCM and related matrices. In addition to the articles mentioned above there are only a few publication that provides information about the eigenvalues (see e.g. \cite{B, HE}).

\newpage
\section{Preliminaries}

Let $A_n^{\alpha,\beta}$ denote the $n\times n$ matrix, which $ij$ entry is given as
\begin{equation}
(A_n^{\alpha,\beta})_{ij}=(i,j)^\alpha[i,j]^\beta,
\end{equation}
where $\alpha,\beta\in\Rset$. In addition, for every $n\in\Zset^+$ we define the $n\times n$ matrix $E_n$ by
\begin{equation}
(E_n)_{ij}=\left\{ \begin{array}{ll}
1 & \textrm{if}\ j\,|\,i\\
0 & \textrm{otherwise.}\\
\end{array}\right.
\end{equation}
The matrix $E_n$ may be referred to as the incidence matrix of the set $\{1,2,\ldots,n\}$ with respect to the divisibility relation.

Next we define some important arithmetical functions that we need. First of all, let $N^{\alpha-\beta}$ be the function such that $N^{\alpha-\beta}(k)=k^{\alpha-\beta}$ for all $k\in\Zset^+$. In addition, let $J_{\alpha-\beta}$ denote the arithmetical function with 
\begin{equation}
J_{\alpha-\beta}(k)=k^{\alpha-\beta}\prod_{p\,|\,k}\left(1-\frac{1}{p^{\alpha-\beta}}\right)
\end{equation}
for all $k\in\Zset_+.$ This function may be seen as a generalization of the Jordan totient function, and it is easy to see that the function $J_{\alpha-\beta}$ can be written as
\begin{equation}
J_{\alpha-\beta}=N^{\alpha-\beta}*\mu,
\end{equation}
the Dirichlet convolution of $N^{\alpha-\beta}$ and the number-theoretic M\"obius function.
\begin{rem}\label{rejordan}
If $\alpha-\beta>0$, then clearly $J_{\alpha-\beta}(k)>0$ for all $k\in\Zset^+.$
\end{rem}
Before we begin to analyze the eigenvalues of the matrix $A_n^{\alpha,\beta}$ we first need to obtain suitable factorizations for it.

\begin{prop}\label{prop1}
Let $F_n=\mathrm{diag}(1,2,\ldots,n)$ and $D_n=\mathrm{diag}(d_1,d_2,\ldots,d_n),$ where
\begin{equation}
d_i=J_{\alpha-\beta}(i)=(N^{\alpha-\beta}*\mu)(i).
\end{equation}
Then the matrix $A_n^{\alpha,\beta}$ can be written as
\begin{equation}
A_n^{\alpha,\beta}=F_n^\beta E_nD_nE_n^TF_n^\beta.
\end{equation}
\end{prop}

\begin{proof}
Since the $ij$ element of the matrix $E_nD_nE_n^T$ is
\begin{equation}
\sum_{k\,|\,i,j} J_{\alpha-\beta}(k)= \sum_{k\,|\,i,j}(N^{\alpha-\beta}*\mu)(k)=[(N^{\alpha-\beta}*\mu)*\zeta]((i,j))=N^{\alpha-\beta}((i,j))=(i,j)^{\alpha-\beta},
\end{equation}
the $ij$ element of the matrix $F_n^\beta E_nD_nE_n^TF_n^\beta$ is
\begin{equation}
i^\beta(i,j)^{\alpha-\beta}j^\beta=(i,j)^\alpha[i,j]^\beta,
\end{equation}
which is also the $ij$ element of $A_n^{\alpha,\beta}$.
\end{proof}

\begin{rem}\label{redet}
By applying Proposition \ref{prop1} it is easy to see that
\begin{equation}
\det A_n^{\alpha,\beta}=\left(n!\right)^{2\beta}\prod_{k=1}^nJ_{\alpha-\beta}(k)=\left(n!\right)^{2\beta}\prod_{k=1}^n\left(N^{\alpha-\beta}*\mu\right)(k).
\end{equation}
\end{rem}

In case when $\alpha>\beta$ we are able to use a different factorization presented in the following proposition.

\begin{prop}\label{prop2}
Suppose that $\alpha>\beta$. Let $J_{\alpha-\beta}$, $D_n$ and $F_n$ be as in Proposition \ref{prop1}, and let  $B_n$ denote the real $n\times n$ matrix with
\begin{equation}
(B_n)_{ij}=\left\{ \begin{array}{ll}
\sqrt{J_{\alpha-\beta}(j)} & \textrm{if}\ j\,|\,i\\
0 & \textrm{otherwise.}\\
\end{array}\right.
\end{equation}

Then the matrix $A_n^{\alpha,\beta}$ can be written as
\begin{equation}
A_n^{\alpha,\beta}=(F_n^\beta B_n)(F_n^\beta B_n)^T=(F_n^\beta E_nD_n^{\frac{1}{2}})(F_n^\beta E_nD_n^{\frac{1}{2}})^T.
\end{equation}
\end{prop}

\begin{proof}
First we observe that the $ij$ element of $B_nB_n^T$ is equal to
\begin{equation}
\sum_{k\,|\,i,j}J_{\alpha-\beta}(k)=\sum_{k\,|\,i,j}(N^{\alpha-\beta}*\mu)(k)=[(N^{\alpha-\beta}*\mu)*\zeta]((i,j))=N^{\alpha-\beta}((i,j))=(i,j)^{\alpha-\beta}.
\end{equation}
Thus the $ij$ element of $(F_n^\beta B_n)(F_n^\beta B_n)^T=F_n^\beta(B_nB_n^T)F_n^\beta$ is
\begin{equation}
i^\beta(i,j)^{\alpha-\beta}j^\beta=(i,j)^\alpha[i,j]^\beta,
\end{equation}
which is also the $ij$ element of the matrix $A_n^{\alpha,\beta}$. Thus we have proven the first equality. The second equality follows from the fact that the matrix $B_n$ can be written as $B_n=E_nD_n^{\frac{1}{2}}.$
\end{proof}

In order to obtain bounds for the eigenvalues of the matrix $A_n^{\alpha,\beta}$ we find out the eigenvalues of the matrix $E_n^TE_n$ for different $n\in\Zset^+.$ The smallest eigenvalue of this matrix is denoted by $t_n$ and the largest by $T_n$. Table \ref{table} shows the values of the constants $t_n$ and $T_n$ for small values of $n$. The $ij$ element of the matrix $E_n^TE_n$ is in fact equal to
\begin{equation}
\left|\{k\in\Zset^+\ \big|\ k\leq n,\ i\,|\,k\ \text{and}\ j\,|\,k\}\right|=\left\lfloor\frac{n}{[i,j]}\right\rfloor,
\end{equation}
the greatest integer that is less than or equal to $\frac{n}{[i,j]}$. This same matrix is also studied by Bege \cite{Be} when he considers it as an example.

\begin{table}[ht!]
	\centering
	
		\begin{tabular}{|c|c|c|c|c|c|c|c|c|}
			\hline
    $n$  & $t_n$ & $T_n$ & $n$  & $t_n$ & $T_n$ & $n$  & $t_n$ & $T_n$\\
    \hline
    $2$ & $0.381966$ & $2.61803$ & $15$ & $0.0616080$ & $23.6243$ & $28$ & $0.0411874$ & $47.1773$\\
    $3$ & $0.267949$ & $3.73205$ & $16$ & $0.0616079$ & $26.1117$ & $29$ & $0.0401315$ & $47.7330$\\
    $4$ & $0.252762$ & $5.78339$ & $17$ & $0.0591935$ & $26.70841$ & $30$ & $0.0343360$ & $51.4915$\\
    $5$ & $0.204371$ & $6.60665$ & $18$ & $0.0584344$ & $29.8007$ & $31$ & $0.0336797$ & $52.0305$\\
    $6$ & $0.129425$ & $9.21230$ & $19$ & $0.0562263$ & $30.3787$ & $32$ & $0.0336797$ & $54.6056$\\
    $7$ & $0.118823$ & $9.92035$ & $20$ & $0.0550575$ & $33.2123$ & $33$ & $0.0322295$ & $55.7392$\\
    $8$ & $0.118764$ & $12.2892$ & $21$ & $0.0505600$ & $34,4522$ & $34$ & $0.0306762$ & $57.2482$\\
    $9$ & $0.116597$ & $13.4520$ & $22$ & $0.0466545$ & $36.0618$ & $35$ & $0.0295618$ & $58.2226$\\
    $10$ & $0.0930874$ & $15.4428$ & $23$ & $0.0452547$ & $36.6470$ & $36$ & $0.0295298$ & $62.7258$\\
    $11$ & $0.087262$ & $16.113$ & $24$ & $0.0452214$ & $41.0878$ & $37$ & $0.0289990$ & $63.2500$\\
    $12$ & $0.087262$ & $16.113$ & $25$ & $0.0451569$ & $41.8465$ & $38$ & $0.0277260$ & $64.7226$\\
    $13$ & $0.0791480$ & $20.4160$ & $26$ & $0.0419049$ & $43.3920$ & $39$ & $0.0267584$ & $65.8548$\\
    $14$ & $0.0681283$ & $22.1909$ & $27$ & $0.0419033$ & $44.6343$ & $40$ & $0.0267526$ & $69.2188$\\
    \hline
		\end{tabular}
		\caption{The constants $t_n$ and $T_n$ for $n\leq40$.}\label{table}
\end{table}

As can be seen from Table \ref{table}, the sequences $(t_n)_{n=1}^\infty$ and $(T_n)_{n=1}^\infty$ seem to possess certain monotonic behavior. This encourages us to present the following conjecture.

\begin{conj}\label{conj}
For every $n\in\Zset^+$ we have
\begin{equation}
t_{n+1}\leq t_n\quad\text{and}\quad T_n\leq T_{n+1}.
\end{equation}
\end{conj}
Calculations show that this conjecture holds for $n=2,\ldots,100$.
\section{Estimations for the eigenvalues}

First we assume that $\alpha>\beta$. From Proposition \ref{prop2} it follows that in this case the matrix $A_n^{\alpha,\beta}$ is positive definite, and thus we are able to give a lower bound for the smallest eigenvalue of $A_n^{\alpha,\beta}$.  

\begin{thm}\label{thm1}
Let $\alpha>\beta$ and let $\lambda_1^{n,\alpha,\beta}$ denote the smallest eigenvalue of the matrix $A_n^{\alpha,\beta}$. Then
\begin{equation}
\lambda_1^{n,\alpha,\beta}\geq t_n\cdot \min_{1\leq i\leq n}J_{\alpha-\beta}(i)\cdot\min\{1,n^{2\beta}\}>0.
\end{equation}
\end{thm}

\begin{proof}
By applying Proposition \ref{prop2} we have
\begin{equation}
A_n^{\alpha,\beta}=(F_n^\beta E_nD_n^{\frac{1}{2}})(F_n^\beta E_nD_n^{\frac{1}{2}})^T.
\end{equation}
By applying Remark \ref{rejordan} and \ref{redet} we deduce that $\det A_n^{\alpha,\beta}\neq0$ and furthermore that $A_n^{\alpha,\beta}$ is invertible. Thus the matrices $A_n^{\alpha,\beta}$ and $(A_n^{\alpha,\beta})^{-1}$ are real symmetric and positive definite and therefore the greatest eigenvalue of $(A_n^{\alpha,\beta})^{-1}$ is also the inverse of the smallest eigenvalue of $A_n^{\alpha,\beta}$. In addition, the greatest eigenvalue of $(A_n^{\alpha,\beta})^{-1}$ is equal to $\matrixnorm{(A_n^{\alpha,\beta})^{-1}}_S,$ the spectral norm of the matrix $(A_n^{\alpha,\beta})^{-1}$. Thus
\begin{equation}
\lambda_1^{n,\alpha,\beta}=\frac{1}{\matrixnorm{(A_n^{\alpha,\beta})^{-1}}_S}=\frac{1}{\matrixnorm{[(F_n^\beta E_nD_n^{\frac{1}{2}})(F_n^\beta E_nD_n^{\frac{1}{2}})^T]^{-1}}_S}.
\end{equation}
By applying the submultiplicativity of the spectral norm we obtain
\begin{align}\label{eq:norm}
&\matrixnorm{[(F_n^\beta E_nD_n^{\frac{1}{2}})(F_n^\beta E_nD_n^{\frac{1}{2}})^T]^{-1}}_S=\matrixnorm{(F_n^\beta)^{-1}(E_n^T)^{-1}D_n^{-1}E_n^{-1}(F_n^\beta)^{-1}}_S\\
&\leq \matrixnorm{(F_n^\beta)^{-1}}_S^2\cdot\left(\matrixnorm{E_n^{-1}}_S\cdot\matrixnorm{(E_n^{-1})^T}_S\right)\cdot\matrixnorm{D_n^{-1}}_S=\matrixnorm{(F_n^\beta)^{-1}}_S^2\cdot\matrixnorm{(E_n^TE_n)^{-1}}_S\cdot\matrixnorm{D_n^{-1}}_S.\notag
\end{align}
Since $J_{\alpha-\beta}(i)>0$ for all $i=1,\ldots,n$ we have
\begin{align}\label{eq:D}
\matrixnorm{D_n^{-1}}_S&=\matrixnorm{\text{diag}\left(\frac{1}{(J_{\alpha-\beta}(1)},\frac{1}{J_{\alpha-\beta}(2)},\ldots,\frac{1}{J_{\alpha-\beta}(n)}\right)}_S\notag\\
&=\max_{1\leq i\leq n}\frac{1}{J_{\alpha-\beta}(i)}=\frac{1}{\min_{1\leq i\leq n}J_{\alpha-\beta}(i)},
\end{align}
and similarly
\begin{equation}\label{eq:F}
\matrixnorm{(F_n^\beta)^{-1}}_S^2=\matrixnorm{\text{diag}\left(\frac{1}{1^\beta},\frac{1}{2^\beta},\ldots,\frac{1}{n^\beta}\right)}_S^2
=\max_{1\leq i\leq n}\frac{1}{i^{2\beta}}=\frac{1}{\min_{1\leq i\leq n}i^{2\beta}}=\frac{1}{\min\{1,n^{2\beta}\}}.
\end{equation}
For the spectral norm of the matrix $(E_n^TE_n)^{-1}$ we have
\begin{equation}\label{eq:E}
\matrixnorm{(E_n^TE_n)^{-1}}_S=\frac{1}{t_n}.
\end{equation}
Now by combining equations \eqref{eq:D}, \eqref{eq:F} and \eqref{eq:E} with \eqref{eq:norm} we obtain
\begin{align}
\lambda_1^{n,\alpha,\beta}&=\frac{1}{\matrixnorm{[(F_n^\beta E_nD_n^{\frac{1}{2}})(F_n^\beta E_nD_n^{\frac{1}{2}})^T]^{-1}}_S}\geq\frac{1}{\matrixnorm{(F_n^\beta)^{-1}}_S^2\cdot\matrixnorm{(E_n^TE_n)^{-1}}_S\cdot\matrixnorm{D_n^{-1}}_S}\notag\\
&=t_n\cdot \min_{1\leq i\leq n}J_{\alpha-\beta}(i)\cdot\min\{1,n^{2\beta}\},
\end{align}
which completes the proof.
\end{proof}

\begin{rem}\label{remalpha}
For $\alpha-\beta\geq1$ we have $\min_{1\leq i\leq n}J_{\alpha-\beta}(i)=1$. In addition, if $\beta\geq0$, then $\min\{1,n^{2\beta}\}=1$ and we simply have
\begin{equation}
\lambda_1^{n,\alpha,\beta}\geq t_n.
\end{equation}
In particular, this holds for the so called power GCD matrix $A_n^{\alpha,\beta}$ in which $\beta=0$ and $\alpha>1$ and for the matrix $A_n^{1,0}$, which is the usual GCD matrix of the set $\{1,2,\ldots,n\}$.

On the other hand, if $\beta<0$, then $\min\{1,n^{2\beta}\}=n^{2\beta}$ and
\begin{equation}
\lambda_1^{n,\alpha,\beta}\geq t_n\cdot n^{2\beta}.
\end{equation}
For example, when considering the so called reciprocal matrix $A_n^{1,-1}$, Theorem \ref{thm1} yields this bound.
\end{rem}

\begin{ex}
Let $n=6$, $\alpha=2$ and $\beta=\frac{1}{2}$. Then we have
\begin{equation}
A_6^{2,\frac{1}{2}}=\left[ \begin{array}{cccccc}
1 & \sqrt{2} & \sqrt{3} & 2 & \sqrt{5} & \sqrt{6} \\
\sqrt{2} & 4\sqrt{2} & \sqrt{6} & 8 & \sqrt{10} & 4\sqrt{6} \\
\sqrt{3} & \sqrt{6} & 9\sqrt{3} & 2\sqrt{3} & \sqrt{15} & 9\sqrt{6} \\
2 & 8 & 2\sqrt{3} & 32 & 2\sqrt{5} & 8\sqrt{3} \\
\sqrt{5} & \sqrt{10} & \sqrt{15} & 2\sqrt{5} & 25\sqrt{5} & \sqrt{30} \\
\sqrt{6} & 4\sqrt{6} & 9\sqrt{6} & 8\sqrt{3} & \sqrt{30} & 36\sqrt{6} \\
\end{array} \right],
\end{equation}
and by Theorem \ref{thm1} and Remark \ref{remalpha} we have $\lambda_1^{6,2,\frac{1}{2}}\geq t_6\approx0.129425$. Direct calculation shows that in fact $\lambda_1^{6,2,\frac{1}{2}}\approx 0.459959$.
\end{ex}

\begin{ex}
Let $n=5$, $\alpha=-2$ and $\beta=-3$. This time we have
\begin{equation}
A_5^{-2,-3}=\left[ \begin{array}{ccccc}
1 & \frac{1}{8} & \frac{1}{27} & \frac{1}{64} & \frac{1}{125} \\
\frac{1}{8} & \frac{1}{32} & \frac{1}{216} & \frac{1}{256} & \frac{1}{1000} \\
\frac{1}{27} & \frac{1}{216} & \frac{1}{243} & \frac{1}{1728} & \frac{1}{3375} \\
\frac{1}{64} & \frac{1}{256} & \frac{1}{1728} & \frac{1}{1024} & \frac{1}{8000} \\
\frac{1}{125} & \frac{1}{1000} & \frac{1}{3375} & \frac{1}{8000} & \frac{1}{7776} \\
\end{array} \right],
\end{equation}
$\min_{1\leq i\leq n}J_1(i)=1$ and $\min\{1,5^{2\cdot(-3)}\}=\frac{1}{15625}$. Thus by Theorem \ref{thm1} we have
$$\lambda_1^{5,-2,-3}\geq t_5\cdot1\cdot\frac{1}{15625}\approx1.30797\cdot10^{-5},$$ although a direct calculation gives $\lambda_1^{5,-2,-3}\approx6.45967\cdot10^{-5}$.
\end{ex}

In Theorem \ref{thm1} we assume that $\alpha>\beta.$ Next we are going to prove a more robust theorem which can be used in any circumstances, but as a downside it also gives a bit more broad bounds for the eigenvalues of the matrix $A_n^{\alpha,\beta}$.

\begin{thm}\label{thm2}
Every eigenvalue of the matrix $A_n^{\alpha,\beta}$ lies in the real interval
\begin{equation}
\Big[2\min\{1,n^{\alpha+\beta}\}-T_n\cdot\max\{1,n^{2\beta}\}\cdot\max_{1\leq i\leq n}|J_{\alpha-\beta}(i)|\ ,\ T_n\cdot\max\{1,n^{2\beta}\}\cdot\max_{1\leq i\leq n}|J_{\alpha-\beta}(i)|\Big].
\end{equation}
\end{thm}

\begin{proof}
Let the matrices $E_n$, $D_n$ and $F_n$ be as above. In addition, we denote
\begin{equation}
\Lambda_n=|D_n|^{\frac{1}{2}}=\text{diag}\left(\sqrt{|J_{\alpha-\beta}(1)|},\sqrt{|J_{\alpha-\beta}(2)|},\ldots,\sqrt{|J_{\alpha-\beta}(n)|}\right).
\end{equation}
By applying Proposition \ref{prop1} we obtain
\begin{equation}
A_n^{\alpha,\beta}=F_n^\beta E_nD_nE_n^TF_n^\beta,
\end{equation}
and next we observe that
\begin{equation}
0_{n\times n}\leq A_n^{\alpha,\beta}\leq F_n^\beta E_n|D_n|E_n^TF_n^\beta=F_n^\beta E_n\Lambda_n\Lambda_n^TE_n^T(F_n^\beta)^T=(F_n^\beta E_n\Lambda_n)(F_n^\beta E_n\Lambda_n)^T,
\end{equation}
where $\leq$ is understood componentwise. By Theorem 8.2.12 in \cite{HJ} we know that now every eigenvalue of $A_n^{\alpha,\beta}$ lies in the region
\begin{equation}
\bigcup_{k=1}^n\big\{z\in\Cset\ \Big|\ |z-k^{\alpha+\beta}|\leq \rho((F_n^\beta E_n\Lambda_n)(F_n^\beta E_n\Lambda_n)^T)-k^{\alpha+\beta}\big\},
\end{equation}
where $\rho((F_n^\beta E_n\Lambda_n)(F_n^\beta E_n\Lambda_n)^T)$ is the spectral radius of the matrix $(F_n^\beta E_n\Lambda_n)(F_n^\beta E_n\Lambda_n)^T$. Since the matrix $(F_n^\beta E_n\Lambda_n)(F_n^\beta E_n\Lambda_n)^T$ is clearly positive semidefinite, we have
\begin{align}
\rho((F_n^\beta E_n\Lambda_n)(F_n^\beta E_n\Lambda_n)^T)=\matrixnorm{(F_n^\beta E_n\Lambda_n)(F_n^\beta E_n\Lambda_n)^T}_S
\leq \matrixnorm{F^\beta}^2_S\cdot\matrixnorm{E_n^TE_n}_S\cdot\matrixnorm{\Lambda_n\Lambda_n^T}_S\notag\\
=T_n\cdot\max_{1\leq i\leq n}i^{2\beta}\cdot\max_{1\leq i\leq n}|J_{\alpha-\beta}(i)|=T_n\cdot\max\{1,n^{2\beta}\}\cdot\max_{1\leq i\leq n}|J_{\alpha-\beta}(i)|.
\end{align}
The matrix $A_n^{\alpha,\beta}$ is real and symmetric which means that all its eigenvalues are real. So we have proven that every eigenvalue of $A_n^{\alpha,\beta}$ lies in the region
\begin{equation}
\bigcup_{k=1}^n\big\{z\in\Rset\ \Big|\ |z-k^{\alpha+\beta}|\leq T_n\cdot\max\{1,n^{2\beta}\}\cdot\max_{1\leq i\leq n}|J_{\alpha-\beta}(i)|-k^{\alpha+\beta}\big\}.
\end{equation}
By removing the absolute value function we obtain that every eigenvalue of $A_n^{\alpha,\beta}$ lies in the real interval
\begin{equation}
\bigcup_{k=1}^n\big[2k^{\alpha+\beta}-T_n\cdot\max\{1,n^{2\beta}\}\cdot\max_{1\leq i\leq n}|J_{\alpha-\beta}(i)|\ ,\ T_n\cdot\max\{1,n^{2\beta}\}\cdot\max_{1\leq i\leq n}|J_{\alpha-\beta}(i)|\big].
\end{equation}
It is obvious that the minimum of the lower bounds of these intervals is obtained in the case when the term $2k^{\alpha+\beta}$ is smallest. Since the sequence $(k^{\alpha+\beta})_{k=1}^\infty$ is monotone, the smallest value is obtained either when $k=1$ or when $k=n$. Thus either the first or the last interval contains all the other intervals, which means that the union of the $n$ intervals is the interval
\begin{equation}
\Big[2\min\{1,n^{\alpha+\beta}\}-T_n\cdot\max\{1,n^{2\beta}\}\cdot\max_{1\leq i\leq n}|J_{\alpha-\beta}(i)|\ ,\ T_n\cdot\max\{1,n^{2\beta}\}\cdot\max_{1\leq i\leq n}|J_{\alpha-\beta}(i)|\Big].
\end{equation}
Thus we have proven the claim.
\end{proof}

\begin{rem}
Theorem \ref{thm2} is not very useful in case when $\beta>0$, since in this case the term $\max\{1,n^{2\beta}\}$ often becomes large.

\end{rem}

\begin{ex}
Let $n=4$, $\alpha=-1$ and $\beta=-1$. Then we obtain
\begin{equation}
A_4^{-1,-1}=\left[ \begin{array}{cccc}
1 & \frac{1}{2} & \frac{1}{3} & \frac{1}{4} \\
\frac{1}{2} & \frac{1}{4} & \frac{1}{6} & \frac{1}{8} \\
\frac{1}{3} & \frac{1}{6} & \frac{1}{9} & \frac{1}{12} \\
\frac{1}{4} & \frac{1}{8} & \frac{1}{12} & \frac{1}{16} \\
\end{array} \right].
\end{equation}
Now $\min\{1,4^{(-1)+(-1)}\}=\frac{1}{16}$, $\max\{1,4^{2\cdot(-1)}\}=1$, $\max_{1\leq i\leq 4}|J_0(i)|=|J_0(1)|=1$ and thus by Theorem \ref{thm2} we know that the eigenvalues of $A_4^{-1,-1}$ lie in the interval
\begin{equation}
\Big[2\cdot\frac{1}{16}-T_4\cdot1\cdot1\ ,\ T_4\cdot1\cdot1\Big]\subset[-5.66,5.79].
\end{equation}
Direct calculation shows that this really is the case, since $A_4^{-1,-1}$ has $0$ as an eigenvalue of multiplicity $3$ and the only nonzero eigenvalue is $1.42361$.
\end{ex}

The following corollary is a direct consequece of Theorem \ref{thm2}.

\begin{cor}
If $\lambda$ is an eigenvalue of the matrix $A_n^{\alpha,\beta}$, then
\begin{equation}
|\lambda|\leq T_n\cdot\max\{1,n^{2\beta}\}\cdot\max_{1\leq i\leq n}|J_{\alpha-\beta}(i)|.
\end{equation}
\end{cor}
\bigskip

\noindent{\bf Acknowledgements} We thank Jesse Railo for Scilab calculations concerning Conjecture \ref{conj}. We also wish to thank the referee for valuable remarks which helped to improve this paper.

\newpage

\end{document}